\documentclass[12pt]{amsart}

\usepackage{amssymb}

\usepackage{enumerate}

\usepackage{graphicx}

\makeatletter
\@namedef{subjclassname@2010}{%
  \textup{2010} Mathematics Subject Classification}
\makeatother

\newtheorem{theorem}{Theorem}[section]
\newtheorem{lemma}[theorem]{Lemma}
\newtheorem{proposition}[theorem]{Proposition}

\theoremstyle{definition}
\newtheorem{definition}[theorem]{Definition}
\newtheorem{example}[theorem]{Example}
\newtheorem{remark}[theorem]{Remark}

\newtheorem*{ack}{Acknowledgement}

\numberwithin{equation}{section}

\DeclareMathOperator{\supp}{supp}
\DeclareMathOperator{\diam}{diam}

\begin{document}

%%%%% To ease editing, for IMPAN journals add:

% \baselineskip=17pt

%%%%%%%%%%%

%% In the running head, replace first names by initials
%% and give an abbreviation of the title.

\title[The stability of Markov operators]{On some results on the stability
of Markov operators}

\author[S. W\c{e}drychowicz]{Stanis\l aw W\c{e}drychowicz}
\address{Department of Mathematics, Rzesz\'{o}w University of Technology,
Al. Powsta\'{n}c\'{o}w Warszawy 12, 35-959 Rzesz\'{o}w, Poland}
\email{swedrych@prz.edu.pl}

\author[A. Wi\'{s}nicki]{Andrzej Wi\'{s}nicki}
\address{Department of Mathematics, Rzesz\'{o}w University of Technology,
Al. Powsta\'{n}c\'{o}w Warszawy 12, 35-959 Rzesz\'{o}w, Poland}
\email{awisnicki@prz.edu.pl}

\date{}

\begin{abstract}
We formulate a criterion for the existence of an invariant measure for a
Feller semigroup defined on a metric space with the e-property for bounded
continuous functions and use it to prove the asymptotic stability of a
semigroup satisfying a lower bound condition. Our results complement those
of A. Lasota and J. A. Yorke in proper metric spaces and of T. Szarek in
Polish spaces.
\end{abstract}



\maketitle

\section{Introduction}

This paper is\ motivated by the study of asymptotic stability of Markov
operators and the work of T. Szarek initiated in \cite{Sz1}. He managed to
extend the Lasota-Yorke lower bound technique \cite{LaYo} from proper metric
spaces to Polish spaces. The crucial difficulty is to establish the
existence of an invariant measure. If bounded sets are relatively compact
then it follows from the Riesz representation theorem and the result of S.
R. Foguel \cite{Fo}. However in Polish spaces, Szarek's ingenious proof,
based on the concept of tightness, is quite a delicate matter.

We show that it is possible to develop the lower bound technique in any
metric space. Our main new tool is the weak sequential completeness of the
weak topology (i.e., weak$^{\ast }$ topology restricted to countably
additive probability measures). Inspired by the Lasota-York theorem \cite[%
Theorem 4.1]{LaYo} and refining its proof, we are able to formulate a
criterion for the existence of an invariant measure for a Markov-Feller
semigroup with the e-property for bounded continuous functions, and prove it
by extracting a Cauchy sequence which, due to sequential completeness of the
weak topology, leads to the existence of an invariant measure. Notice that
in Polish spaces, our criterion is weaker than Szarek's one (see \cite[%
Proposition 2.1]{Sz4}, \cite[Theorem 3.1]{LaSz}) and its strengthening in
\cite[Proposition 2]{BKS}, yet strong enough to prove the stability results
in the spirit of \cite[Theorem 3.3]{Sz1} which extend \cite[Theorem 3.3]{Sz4}%
, \cite[Corollary 5.4]{SzWo} and \cite[Theorem 3.3]{CzHo}.\ It seems that
the general case of metric spaces has not been carefully studied yet, though
there are some natural spaces which are nonseparable or noncomplete. For
example, the Skorokhod space $D[0,1]$ with the uniform topology is not
separable, and equipped with the metric%
\begin{equation*}
d(x,y)=\inf_{\lambda \in \Lambda }(\sup_{t\in \lbrack 0,1]}\left\vert
\lambda (t)-t\right\vert +\sup_{t\in \lbrack 0,1]}\left\vert x(\lambda
(t))-y(t)\right\vert ),
\end{equation*}%
is not complete.

The paper is organized as follows. Section 2 contains some notation from the
theory of Markov semigroups. In Section 3 we prove a criterion for the
existence of an invariant measure for a Feller semigroup, which is used in
Section 4 to prove the stability results. An application to iterated
function systems and jump processes, similar to \cite{CzHo, Sz4}, is also
given.

\section{Markov operators}

Let $(X,\rho )$ be a metric space. We write $B(x,r)$ for the open ball with
a centre $x$ and radius $r$, $\mathcal{B}\left( X\right) $ for the space of
Borel subsets of $X$ and $B_{b}(X)$ for the space of (real) bounded Borel
measurable functions with the supremum norm $\left\Vert \cdot \right\Vert
_{\infty }$. Let $C_{b}(X)$ denote its subspace of bounded continuous
functions and $\mathrm{Lip}_{b}(X)$ the subspace of bounded Lipschitz
functions. Let $\mathcal{M}_{1}(X)$ (or $\mathcal{M}_{1}$ for short) be the
space of Borel probability measures on $X$. For $\varphi \in B_{b}(X)$ and $%
\mu \in \mathcal{M}_{1}$ we use the notation $\left\langle \varphi ,\mu
\right\rangle =\int_{X}\varphi (x)\mu (dx).$

Let $(P_{t})_{t\geq 0}$ be a Markov semigroup defined on $B_{b}(X)$. Thus $%
P_{t}\mathbf{1}_{X}=\mathbf{1}_{X}$ for each $t\geq 0$ and $P_{t}\varphi
\geq 0$ if $\varphi \geq 0$. Throughout this paper we shall assume that the
semigroup $(P_{t})_{t\geq 0}$ is Feller, i.e., $P_{t}(C_{b}(X))\subset
C_{b}(X)$ and there exists a semigroup $P_{t}^{\ast }:\mathcal{M}%
_{1}\rightarrow \mathcal{M}_{1},t\geq 0,$ dual to $(P_{t})_{t\geq 0},$ i.e.,
$\left\langle P_{t}\varphi ,\mu \right\rangle =\left\langle \varphi
,P_{t}^{\ast }\mu \right\rangle $ for every $\varphi \in B_{b}(X)$, $\mu \in
\mathcal{M}_{1}$ and $t\geq 0.$ We shall also assume that $(P_{t})_{t\geq 0}$
is stochastically continuous, i.e., $\lim_{t\rightarrow 0^{+}}P_{t}\psi
(x)=\psi (x)$ for $\psi \in C_{b}(X)$ and $x\in X.$

Recall that $\mu _{\ast }\in \mathcal{M}_{1}$ is invariant for the semigroup
$(P_{t})_{t\geq 0}$ if $P_{t}^{\ast }\mu _{\ast }=\mu _{\ast }$ for all $%
t\geq 0.$ For a given $t>0$ and $\mu \in \mathcal{M}_{1}$, define $Q_{t}\mu
=t^{-1}\int_{0}^{t}P_{s}^{\ast }\mu \ ds.$ We write $Q_{t}(x,\cdot
):=Q_{t}\delta _{x}.$ A semigroup $(P_{t})_{t\geq 0}$ is said to be
asymptotically stable if there exists a unique invariant measure $\mu _{\ast
}\in \mathcal{M}_{1}$ such that $(P_{t}^{\ast }\mu )_{t\geq 0}$ converges
weakly to $\mu _{\ast }$ as $t\rightarrow \infty $ for every $\mu \in
\mathcal{M}_{1}$.

\begin{definition}
We say that a semigroup $(P_{t})_{t\geq 0}$ has the e-property if for any
bounded and Lipschitz function $\varphi $ the family of functions $%
(P_{t}\varphi )_{t\geq 0}$ is equicontinuous at every point $x$ of $X$ ,
i.e., for any $\varphi \in \mathrm{Lip}_{b}(X)$ and $x\in X$ we have%
\begin{equation*}
\lim_{y\rightarrow x}\sup_{t\geq 0}\left\vert P_{t}\varphi (y)-P_{t}\varphi
(x)\right\vert =0.
\end{equation*}
\end{definition}

The e-property was introduced by Szarek as a counterpart of the e-chain
property defined in \cite[p.144]{MeTw}. As we work in a general metric
space, we will also consider a little stronger, e-property for bounded
continuous functions (see also \cite{Cz,CzHo}). For example, the semigroup%
\begin{equation*}
P_{t}\varphi (x)=\varphi (x+t),\ \varphi \in B_{b}(\mathbb{R}),
\end{equation*}%
has the e-property but there exists a bounded continuous function $\psi :%
\mathbb{R}\rightarrow \mathbb{R}$ such that $(P_{t}\psi )_{t\geq 0}$ is not
equicontinuous at any point $z\in \mathbb{R}$.

\section{Invariant measures}

The following criterion of `weak convergence' is basic for our approach to
construct an invariant measure, see \cite[Theorem 8.7.1]{Bo}, \cite[Theorem 1%
]{Du1}.

\begin{theorem}
\label{complete}Let $X$ be a metric space. The weak convergence on $\mathcal{%
M}_{1}(X)$ is weak sequentially complete, i.e., if a sequence of probablity
measures $(\mu _{n})$ satisfies $\lim_{n,m\rightarrow \infty }\left\vert
\left\langle \varphi ,\mu _{n}\right\rangle -\left\langle \varphi ,\mu
_{m}\right\rangle \right\vert =0$ for every $\varphi \in C_{b}(X),$ then it
converges weakly to some $\mu \in \mathcal{M}_{1}(X).$
\end{theorem}

For a given integer $k\geq 1,t_{1},...,t_{k}\geq 0$ and $\mu \in \mathcal{M}%
_{1}$ we let $Q_{t_{k},t_{k-1},...,t_{1}}\mu =Q_{t_{k}}...Q_{t_{1}}\mu .$ We
will use the following two simple lemmas. Let $\left\Vert \cdot \right\Vert
_{\mathrm{TV}}$ denotes the total variation norm (which agrees with operator
norm).

\begin{lemma}
\label{QT}For given $k\geq 1,t_{1},...,t_{k}\geq 0,$%
\begin{equation*}
\lim_{T\rightarrow \infty }\sup_{\mu \in \mathcal{M}_{1}}\left\Vert
Q_{T,t_{k},...,t_{1}}\mu -Q_{T}\mu \right\Vert _{\mathrm{TV}}=0.
\end{equation*}
\end{lemma}

\begin{proof}
See, e.g., \cite[Lemma 2]{KPS}.
\end{proof}

\begin{lemma}
\label{ineq}For any $\alpha \in (0,1),$ $k>0$ and sufficiently small $%
\varepsilon >0$ (depending on $\alpha $ and $k$),%
\begin{equation*}
\frac{1-\alpha (1+\varepsilon )}{1-\alpha (1-\sqrt[k]{\varepsilon })^{k}}>1-%
\sqrt[k+1]{\varepsilon }.
\end{equation*}
\end{lemma}

\begin{proof}
By a direct calculation, using the formula $(1-(1-\sqrt[k]{\varepsilon }%
)^{k})/\sqrt[k+1]{\varepsilon }\rightarrow 0$ when $\varepsilon \rightarrow
0^{+}.$
\end{proof}

The following proposition is a key result for our analysis. Its proof is
partly inspired by \cite[Theorem 4.1]{LaYo}.

\begin{proposition}
\label{inv}Assume that there exists $z\in X$ such that for every $%
\varepsilon >0$%
\begin{equation*}
\liminf_{t\rightarrow \infty }\frac{1}{t}\int_{0}^{t}P_{s}^{\ast }\delta
_{z}(B(z,\varepsilon ))ds>0.
\end{equation*}%
If the family $(P_{t}\varphi )_{t\geq 0}$ is equicontinuous at $z$ for every
$\varphi \in C_{b}(X),$ then the semigroup $(P_{t})_{t\geq 0}$ has an
invariant probability measure.
\end{proposition}

\begin{proof}
Fix $\varphi \in C_{b}(X)$ and $\bar{\varepsilon}>0$. By assumption, there
exists $\delta >0$ such that%
\begin{equation*}
\left\vert P_{t}\varphi (x)-P_{t}\varphi (y)\right\vert <\bar{\varepsilon}
\end{equation*}%
for $t\in \lbrack 0,\infty )$ and $x,y\in B(z,\delta ).$ Let%
\begin{equation*}
\alpha =\liminf_{t\rightarrow \infty }\sup_{x\in B(z,\delta )}Q_{t}\delta
_{x}(B(z,\delta ))>0.
\end{equation*}%
We may assume that $\alpha <1.$ Select a positive integer $K$ such that $%
2(1-\alpha )^{K}\left\Vert \varphi \right\Vert _{\infty }<\bar{\varepsilon}.$
Then there exists $0<\varepsilon <1$ such that $2(1-\alpha (1-\sqrt[K]{%
\varepsilon })^{K})^{K}\left\Vert \varphi \right\Vert _{\infty }<\bar{%
\varepsilon},$ $\frac{1-\varepsilon }{(1-\sqrt[K]{\varepsilon })^{K}}<2$
and, by Lemma \ref{ineq},%
\begin{equation}
\frac{1-\alpha (1+\varepsilon )}{1-(1-\sqrt[k]{\varepsilon })^{k}\alpha }>1-%
\sqrt[k+1]{\varepsilon }  \label{sep10}
\end{equation}%
for any $k=1,...,K.$ It follows from Lemma \ref{QT} and the definition of $%
\alpha $ that there exists an increasing sequence $(t_{k})\subset (0,\infty )
$ such that%
\begin{eqnarray}
\sup_{x\in B(z,\delta )}Q_{T,t_{k},t_{k-1},...,t_{j}}\delta _{x}(B(z,\delta
)) &>&\alpha (1-\frac{\varepsilon }{3}),  \label{sep11} \\
Q_{t_{k},t_{k-1},...,t_{j}}\delta _{y}(B(z,\delta )) &<&\alpha
(1+\varepsilon ),  \label{sep12} \\
\left\Vert Q_{T,t_{k},...,t_{j}}\mu -Q_{T}\mu \right\Vert _{\mathrm{TV}}
&<&\min \{\frac{\bar{\varepsilon}}{K},\frac{\alpha \varepsilon }{3}\}
\label{sep125}
\end{eqnarray}%
for any $y\in B(z,\delta ),$ $\mu \in \mathcal{M}_{1}$, $k\geq 1,1\leq j\leq
k$ and $T\geq t_{k+1}.$

Fix positive integers $n_{1},n_{2}\geq t_{K+1}.$ Then there exist $%
x_{n_{1}},x_{n_{2}}\in B(z,\delta )$ such that for any $k\leq K$ and $1\leq
j\leq k,$
\begin{equation}
Q_{n_{i},t_{k},t_{k-1},...,t_{j}}\delta _{x_{n_{i}}}(B(z,\delta ))>\alpha
(1-\varepsilon ),\ i=1,2.  \label{sep111}
\end{equation}%
We show that
\begin{equation*}
\left\vert \left\langle \varphi ,Q_{n_{1}}\delta _{x_{n_{1}}}\right\rangle
-\left\langle \varphi ,Q_{n_{2}}\delta _{x_{n_{2}}}\right\rangle \right\vert
<8\bar{\varepsilon}.
\end{equation*}%
To this end, let $\sigma _{k}=(1-\varepsilon )(1-\sqrt{\varepsilon })...(1-%
\sqrt[k]{\varepsilon })\alpha $ and define by induction four finite
sequences of probability measures $(\mu _{1}^{k}),(\mu _{2}^{k}),(\nu
_{1}^{k}),(\nu _{2}^{k}),1\leq k\leq K$ such that $%
%TCIMACRO{\TeXButton{supp}{\supp}}%
%BeginExpansion
\supp%
%EndExpansion
\nu _{i}^{k}\subset B(z,\delta ),1\leq k\leq K,$ and%
\begin{align}
Q_{t_{k},...,t_{1}}Q_{n_{i}}\delta _{x_{n_{i}}}=& \sigma
_{1}Q_{t_{k},...,t_{2}}\nu _{i}^{1}+\sigma _{2}(1-\sigma
_{1})Q_{t_{k},...,t_{3}}\nu _{i}^{2}  \label{sep13} \\
& +...+\sigma _{k}\prod_{s=1}^{k-1}(1-\sigma _{s})\nu
_{i}^{k}+\prod_{s=1}^{k}(1-\sigma _{s})\mu _{i}^{k}  \notag
\end{align}%
for $i=1,2.$ If $k=1,$ we set%
\begin{align*}
\nu _{i}^{1}(A)& =\frac{Q_{t_{1}}Q_{n_{i}}\delta _{x_{n_{i}}}(A\cap
B(z,\delta ))}{Q_{t_{1}}Q_{n_{i}}\delta _{x_{n_{i}}}(B(z,\delta ))}, \\
\mu _{i}^{1}(A)& =\frac{1}{1-\sigma _{1}}(Q_{t_{1}}Q_{n_{i}}\delta
_{x_{n_{i}}}(A)-\sigma _{1}\nu _{i}^{1}(A)),\ i=1,2.
\end{align*}%
Notice that $\mu _{i}^{1},\nu _{i}^{1}\in \mathcal{M}_{1}$ since $%
Q_{t_{1}}Q_{n_{i}}\delta _{x_{n_{i}}}(B(z,\delta ))=Q_{n_{i}}Q_{t_{1}}\delta
_{x_{n_{i}}}(B(z,\delta ))>\sigma _{1}$ by (\ref{sep111}) and $%
%TCIMACRO{\TeXButton{supp}{\supp}}%
%BeginExpansion
\supp%
%EndExpansion
\nu _{i}^{1}\subset B(z,\delta ).$ Clearly, $Q_{t_{1}}Q_{n_{i}}\delta
_{x_{n_{i}}}=(1-\sigma _{1})\mu _{i}^{1}+\sigma _{1}\nu _{i}^{1}.$ Suppose
now that we have chosen $\mu _{1}^{j},\mu _{2}^{j},\nu _{1}^{j},\nu _{2}^{j}$
$(0\leq j\leq k-1)$ in such a way that $\supp\nu _{i}^{j}\subset B(z,\delta
),1\leq j\leq k-1,$ and%
\begin{eqnarray*}
Q_{t_{k-1},...,t_{1}}Q_{n_{i}}\delta _{x_{n_{i}}} &=&\sigma
_{1}Q_{t_{k-1},...,t_{2}}\nu _{i}^{1}+\sigma _{2}(1-\sigma
_{1})Q_{t_{k-1},...,t_{3}}\nu _{i}^{2} \\
&&+...+\sigma _{k-1}\prod_{s=1}^{k-2}(1-\sigma _{s})\nu
_{i}^{k-1}+\prod_{s=1}^{k-1}(1-\sigma _{s})\mu _{i}^{k-1}
\end{eqnarray*}%
for $i=1,2.$ It follows from (\ref{sep12}) that for any $1\leq j,s\leq k-1,$%
\begin{equation*}
Q_{t_{k},t_{k-1},...,t_{s}}\nu _{i}^{j}(B(z,\delta
))=\int_{X}Q_{t_{k},t_{k-1},...,t_{s}}\delta _{x}(B(z,\delta ))\nu
_{i}^{j}(dx)<\alpha (1+\varepsilon ).
\end{equation*}%
Hence%
\begin{align*}
& \prod_{s=1}^{k-1}(1-\sigma _{s})Q_{t_{k}}\mu _{i}^{k-1}(B(z,\delta
))>Q_{t_{k},t_{k-1},...,t_{1}}Q_{n_{i}}\delta _{x_{n_{i}}}(B(z,\delta )) \\
-& \alpha (1+\varepsilon )(\sigma _{1}+\sigma _{2}(1-\sigma _{1})+...+\sigma
_{k-1}\prod_{s=1}^{k-2}(1-\sigma _{s})) \\
>& \alpha (1-\varepsilon )-\alpha (1+\varepsilon )(\sigma _{1}+\sigma
_{2}(1-\sigma _{1})+...+\sigma _{k-1}\prod_{s=1}^{k-2}(1-\sigma _{s})),
\end{align*}%
since $Q_{t_{k},t_{k-1},...,t_{1}}Q_{n_{i}}\delta _{x_{n_{i}}}(B(z,\delta
))=Q_{n_{i}}Q_{t_{k},t_{k-1},...,t_{1}}\delta _{x_{n_{i}}}(B(z,\delta
))>\alpha (1-\varepsilon ),$ by (\ref{sep111}). Now notice that from (\ref%
{sep10}),
\begin{align*}
\frac{\alpha (1-\varepsilon )-\alpha (1+\varepsilon )\sigma _{1}}{1-\sigma
_{1}}& =\sigma _{1}\frac{1-\alpha (1+\varepsilon )}{1-\sigma _{1}}>\sigma
_{1}(1-\sqrt{\varepsilon })=\sigma _{2}, \\
\sigma _{2}\frac{1-\alpha (1+\varepsilon )}{1-\sigma _{2}}& >\sigma _{2}(1-%
\sqrt[3]{\varepsilon })=\sigma _{3}, \\
& ... \\
\sigma _{k-1}\frac{1-\alpha (1+\varepsilon )}{1-\sigma _{k-1}}& >\sigma
_{k-1}(1-\sqrt[k]{\varepsilon })=\sigma _{k}
\end{align*}%
and consequently,%
\begin{equation*}
\prod_{s=1}^{k-1}(1-\sigma _{s})Q_{t_{k}}\mu _{i}^{k-1}(B(z,\delta ))>\sigma
_{k}\prod_{s=1}^{k-1}(1-\sigma _{s})
\end{equation*}%
which gives%
\begin{equation*}
Q_{t_{k}}\mu _{i}^{k-1}(B(z,\delta ))>\sigma _{k}.
\end{equation*}%
Define%
\begin{align*}
\nu _{i}^{k}(A)& =\frac{Q_{t_{k}}\mu _{i}^{k-1}(A\cap B(z,\delta ))}{%
Q_{t_{k}}\mu _{i}^{k-1}(B(z,\delta ))}, \\
\mu _{i}^{k}(A)& =\frac{1}{1-\sigma _{k}}(Q_{t_{k}}\mu _{i}^{k-1}(A)-\sigma
_{k}\nu _{i}^{k}(A)).
\end{align*}%
Thus $\mu _{i}^{k},\nu _{i}^{k}\in \mathcal{M}_{1}$ and $\supp\nu
_{i}^{k}\subset (B(z,\delta ))$ for $i=1,2.$ Furthermore, (\ref{sep13}) is
satisfied which completes the inductive step.

Now, from (\ref{sep13}), we have
\begin{align*}
& \left\vert \left\langle \varphi ,Q_{t_{K},...,t_{1}}Q_{n_{1}}\delta
_{x_{n_{1}}}\right\rangle -\left\langle \varphi
,Q_{t_{K},...,t_{1}}Q_{n_{2}}\delta _{x_{n_{2}}}\right\rangle \right\vert  \\
\leq \,& \sigma _{1}\left\vert \left\langle \varphi ,Q_{t_{K},...,t_{2}}(\nu
_{i}^{1}-\nu _{2}^{1})\right\rangle \right\vert +\sigma _{2}(1-\sigma
_{1})\left\vert \left\langle \varphi ,Q_{t_{K},...,t_{3}}(\nu _{1}^{2}-\nu
_{2}^{2})\right\rangle \right\vert  \\
& +...+\sigma _{k}\prod_{s=1}^{K-1}(1-\sigma _{s})\left\vert \left\langle
\varphi ,\nu _{1}^{K}-\nu _{2}^{K}\right\rangle \right\vert
+\prod_{s=1}^{K}(1-\sigma _{s})\left\vert \left\langle \varphi ,\mu
_{1}^{K}-\mu _{2}^{K}\right\rangle \right\vert  \\
\overset{(\ref{sep125})}{\leq }& \sigma _{1}\left\vert \left\langle \varphi
,Q_{t_{K}}(\nu _{i}^{1}-\nu _{2}^{1})\right\rangle \right\vert +2\frac{\bar{%
\varepsilon}}{K}+\sigma _{2}(1-\sigma _{1})\left\vert \left\langle \varphi
,Q_{t_{K},...,t_{3}}(\nu _{1}^{2}-\nu _{2}^{2})\right\rangle \right\vert +2%
\frac{\bar{\varepsilon}}{K} \\
& +...+\sigma _{k}\prod_{s=1}^{K-1}(1-\sigma _{s})\left\vert \left\langle
\varphi ,\nu _{1}^{K}-\nu _{2}^{K}\right\rangle \right\vert
+\prod_{s=1}^{K}(1-\sigma _{s})\left\vert \left\langle \varphi ,\mu
_{1}^{K}-\mu _{2}^{K}\right\rangle \right\vert  \\
\leq \,& \sigma _{1}\left\vert \left\langle \frac{1}{t_{K}}%
\int_{0}^{t_{K}}P_{s}\varphi \ ds,\nu _{1}^{1}-\nu _{2}^{1}\right\rangle
\right\vert +\sigma _{2}(1-\sigma _{1})\left\vert \left\langle \frac{1}{t_{K}%
}\int_{0}^{t_{K}}P_{s}\varphi \ ds,\nu _{1}^{2}-\nu _{2}^{2}\right\rangle
\right\vert  \\
& +...+\sigma _{k}\prod_{s=1}^{K-1}(1-\sigma _{s})\left\vert \left\langle
\varphi ,\nu _{1}^{K}-\nu _{2}^{K}\right\rangle \right\vert
+2\prod_{s=1}^{K}(1-\sigma _{s})\left\Vert \varphi \right\Vert _{\infty
}+2(K-1)\frac{\bar{\varepsilon}}{K} \\
\leq \,& (\sigma _{1}+\sigma _{1}(1-\sigma _{K-1})+...+\sigma _{1}(1-\sigma
_{K-1})^{K-1})\sup_{x,y\in B(z,\delta ),t\geq 0}\left\vert P_{t}\varphi
(x)-P_{t}\varphi (y)\right\vert  \\
& +2(1-\sigma _{K})^{K}\left\Vert \varphi \right\Vert _{\infty }+2(K-1)\bar{%
\varepsilon}<\frac{\sigma _{1}}{\sigma _{K-1}}\bar{\varepsilon}+\bar{%
\varepsilon}+2(K-1)\frac{\bar{\varepsilon}}{K}<5\bar{\varepsilon}
\end{align*}%
and hence%
\begin{equation*}
\left\vert \left\langle \varphi ,Q_{n_{1}}\delta _{x_{n_{1}}}\right\rangle
-\left\langle \varphi ,Q_{n_{2}}\delta _{x_{n_{2}}}\right\rangle \right\vert
\leq \left\vert \left\langle \varphi ,Q_{t_{K},...,t_{1}}Q_{n_{1}}\delta
_{x_{n_{1}}}\right\rangle -\left\langle \varphi
,Q_{t_{K},...,t_{1}}Q_{n_{2}}\delta _{x_{n_{2}}}\right\rangle \right\vert +%
\frac{2\bar{\varepsilon}}{K}<6\bar{\varepsilon}.
\end{equation*}%
Therefore,%
\begin{align*}
& \left\vert \left\langle \varphi ,Q_{n_{1}}\delta _{z}\right\rangle
-\left\langle \varphi ,Q_{n_{2}}\delta _{z}\right\rangle \right\vert \leq
\left\vert \left\langle \varphi ,Q_{n_{1}}\delta _{z}\right\rangle
-\left\langle \varphi ,Q_{n_{1}}\delta _{x_{n_{1}}}\right\rangle \right\vert
\\
+\,& \left\vert \left\langle \varphi ,Q_{n_{1}}\delta
_{x_{n_{1}}}\right\rangle -\left\langle \varphi ,Q_{n_{2}}\delta
_{x_{n_{2}}}\right\rangle \right\vert +\left\vert \left\langle \varphi
,Q_{n_{2}}\delta x_{n_{2}}\right\rangle -\left\langle \varphi
,Q_{n_{2}}\delta _{z}\right\rangle \right\vert  \\
<\,& \frac{1}{n_{1}}\int_{0}^{n_{1}}\left\vert P_{s}\varphi (z)-P_{s}\varphi
(x_{n_{1}})\right\vert ds+6\bar{\varepsilon}+\frac{1}{n_{2}}%
\int_{0}^{n_{2}}\left\vert P_{s}\varphi (z)-P_{s}\varphi
(x_{n_{2}})\right\vert ds<8\bar{\varepsilon}
\end{align*}%
for $n_{1},n_{2}\geq t_{K+1}.$ Hence%
\begin{equation*}
\lim_{n,m\rightarrow \infty }\left\vert \left\langle \varphi ,Q_{n}\delta
_{z}\right\rangle -\left\langle \varphi ,Q_{m}\delta _{z}\right\rangle
\right\vert \leq 8\bar{\varepsilon}
\end{equation*}%
and letting $\bar{\varepsilon}\rightarrow 0,$ we have%
\begin{equation}
\lim_{n,m\rightarrow \infty }\left\vert \left\langle \varphi ,Q_{n}\delta
_{z}\right\rangle -\left\langle \varphi ,Q_{m}\delta _{z}\right\rangle
\right\vert =0  \label{Cau}
\end{equation}%
which shows that $(Q_{n}\delta _{z})_{n\in \mathbb{N}}$ is a weak Cauchy
sequence. We now conclude from Theorem \ref{complete} that $(Q_{n}\delta
_{z})$ converges weakly to some probability measure $\mu _{\ast }$ which is
invariant for the semigroup $(P_{t})_{t\geq 0}$ since, by a standard
argument,%
\begin{equation*}
\left\Vert P_{t}^{\ast }Q_{n}\delta _{z}-Q_{n}\delta _{z}\right\Vert _{%
\mathrm{TV}}\leq \frac{4t}{n}\rightarrow 0\text{ if }n\rightarrow \infty
\end{equation*}%
for every $t\geq 0,$ and $P_{t}^{\ast }$ is weakly continuous (as the dual
of a Feller operator).
\end{proof}

\begin{remark}
\label{remark}If $X$ is complete and each $Q_{n}\delta _{z}$ is tight, it
follows from \cite[Corollary 8.6.3]{Bo} that for the uniform tightness of $%
(Q_{n}\delta _{z})_{n\in \mathbb{N}}$ it is sufficient to obtain (\ref{Cau})
for every bounded Lipschitz function $\varphi $. Then there exists a
subsequence $(Q_{n_{k}}\delta _{z})_{k\in \mathbb{N}}$ converging weakly to
a $(P_{t})$-invariant probability measure $\mu _{\ast }.$ Therefore, we can
relax the assumption of the e-property for bounded continuous functions in
Proposition \ref{inv} to the (usual) e-property in this case (e.g., in
Polish spaces).
\end{remark}

We can now state our criterion for the existence of the invariant measure of
a Markov semigroup in any metric space (in particular, nonseparable or
noncomplete).

\begin{theorem}
\label{inv2}Let $X$ be a metric space and assume that there exists $z\in X$
and a probability measure $\mu $ such that for every $\varepsilon >0$%
\begin{equation*}
\liminf_{t\rightarrow \infty }\frac{1}{t}\int_{0}^{t}P_{s}^{\ast }\mu
(B(z,\varepsilon ))ds>0.
\end{equation*}%
If $(P_{t}\varphi )_{t\geq 0}$ is equicontinuous at $z$ for every $\varphi
\in C_{b}(X)$, then the semigroup $(P_{t})_{t\geq 0}$ has an invariant
probability measure.
\end{theorem}

\begin{proof}
Consider two cases. If there exists $\bar{\delta}>0$ such that%
\begin{equation*}
\liminf_{t\rightarrow \infty }\sup_{x\in B(z,\delta )}Q_{t}\delta
_{x}(B(z,\delta ))>0
\end{equation*}%
for every $\delta \in (0,\bar{\delta}),$ then we can follow the proof of
Proposition \ref{inv} directly. Thus it suffices to consider the case when
there exists $\bar{\delta}>0$ such that
\begin{equation*}
\liminf_{t\rightarrow \infty }\sup_{x\in B(z,\delta )}Q_{t}\delta
_{x}(B(z,\delta ))=0
\end{equation*}%
for every $\delta \in (0,\bar{\delta}).$ Then we can follow the proof of the
proposition with $\delta _{x_{n_{1}}},\delta _{x_{n_{2}}}$ replaced by $\mu $%
.
\end{proof}

\begin{remark}
As noted in the Introduction, in the case of Polish spaces, our condition is
stronger than that in \cite[Theorem 3.1]{LaSz}, which is in turn stronger
than the condition given in Proposition 2 of \cite{BKS}: there exists $z\in
X $ such that for every $\varepsilon >0,$%
\begin{equation*}
\limsup_{t\rightarrow \infty }\sup_{\mu \in \mathcal{M}_{1}}\frac{1}{t}%
\int_{0}^{t}P_{s}^{\ast }\delta _{z}(B(z,\varepsilon ))ds>0.
\end{equation*}
It results from our method that refines the technique from \cite{LaYo}: it
appears that we have to control eventually the whole sequence $(Q_{n}\mu )$
to be able to select the same subsequence for any $\bar{\varepsilon}>0$. On
the other hand, even in a Polish space, this gives a new, perhaps more
direct way to find an invariant measure that may have further consequences.
\end{remark}

Note that our approach works in the discrete case as well and thus we have
the following theorem (see, e.g., \cite{EtKu} for unexplained notions).

\begin{theorem}
Let $X$ be a metric space and $\pi :X\times \mathcal{B}\left( X\right)
\rightarrow \lbrack 0,1]$ a transition function for a discrete-time Markov
chain $\Phi $. Assume that there exists $z\in X$ and a probability measure $%
\mu $ such that for every $\varepsilon >0$%
\begin{equation*}
\liminf_{n\rightarrow \infty }\left( \frac{1}{n}\sum\limits_{i=1}^{n}%
\int_{X}\pi ^{i}(x,B(z,\varepsilon ))\mu (dx)\right) >0.
\end{equation*}%
If the family $\left\{ \int_{X}\varphi (y)\pi ^{n}(z,dy):n\in \mathbb{N}%
\right\} $ is equicontinuous for every $\varphi \in C_{b}(X)$, then $\Phi $
has an invariant probability measure.
\end{theorem}

Notice that although in Polish spaces our results are weaker than Szarek's
theorem \cite[Proposition 2.1]{Sz4} (see also \cite[Theorem 3.1]{LaSz}),
they are strong enough to be applied analogously in the theory of iterated
function systems and stochastic partial differential equations as the
following example shows (at least when $X$ is complete).

\begin{example}
\label{1}(see \cite[Example, p. 1853]{Sz4}). Let $X$ be a complete metric
space and consider on $X$ an iterated function system
\begin{equation*}
(w,p)_{N}=(w_{1},...,w_{N},p_{1},...,p_{N}).
\end{equation*}%
Thus $w_{i}:X\rightarrow X$ are continuous transformations and $%
p_{i}:X\rightarrow \lbrack 0,1],i=1,...,N,$ are continuous functions that
satisfy $\sum_{i=1}^{N}p_{i}(x)=1$ for $x\in X.$ Furthermore, let $(\Omega ,%
\mathcal{F},\mathbb{P})$ be a probability space and $(\tau _{n})_{n\geq 0}$
be a sequence of random variables $\tau _{n}:\Omega \rightarrow \mathbb{R}%
_{+}$ with $\tau _{0}=0$ and such that $\Delta \tau _{n}=\tau _{n}-\tau
_{n-1}$ are independent and have the same density $\gamma e^{-\gamma t}.$
Given a continuous semi-flow $(S(t))_{t\geq 0}$ on $X,$ we define the Markov
chain $\Phi =(\Phi _{n})_{n\geq 1}$ in the following way: choose $x\in X$
and let $\xi _{1}=S(\tau _{1})(x).$ Next, we randomly select an integer $%
i_{1}$ from the set $\{1,...,N\}$ in such a way that the probablity of
choosing $k$ is $p_{k}(\xi _{1}).$ Set $\Phi _{1}=w_{i_{1}}(\xi _{1}).$
Having $\Phi _{1}$ we define $\xi _{2}=S(\Delta \tau _{2})(\Phi _{1}),$
select $i_{2}$ in a similar way, set $\Phi _{2}=w_{i_{2}}(\xi _{2})$ and so
on. As in \cite{Sz4} we make the following assumptions:

\begin{enumerate}
\item[(i)] there exists $r<1$ such that for $x,y\in X,$%
\begin{equation*}
\sum_{i=1}^{N}p_{i}(x)\rho (w_{i}(x),w_{i}(y))\leq r\rho (x,y),
\end{equation*}

\item[(ii)] there exists $a>0$ such that for $x,y\in X,$%
\begin{equation*}
\sum_{i=1}^{N}\left\vert p_{i}(x)-p_{i}(y)\right\vert \leq a\rho (x,y),
\end{equation*}

\item[(iii)] there exists $\kappa >0$ such that for $x,y\in X$ and $t\geq 0,$%
\begin{equation*}
\rho (S(t)(x),S(t)(y))\leq e^{\kappa t}\rho (x,y).
\end{equation*}
\end{enumerate}

Moreover, we assume that $r+\frac{\kappa }{\gamma }<1$ and the semi-flow $%
(S(t))_{t\geq 0}$ has a (compact) global attractor. Define%
\begin{equation*}
Pf(x)=\sum\limits_{i=1}^{N}\int_{0}^{\infty }\gamma e^{-\gamma
t}p_{i}(S(t)(x))f(w_{i}(S(t)(x)))dt
\end{equation*}%
for $f\in C_{b}(X),x\in X$, and notice that for every $\varepsilon >0$ there
exists $\bar{t}>0$ such that
\begin{equation*}
P^{\ast }\mu (\bigcup\nolimits_{i=1}^{N}\bigcup\nolimits_{t\in \lbrack 0,%
\bar{t}]}w_{i}(S(t)B))>\mu (B)-\varepsilon
\end{equation*}%
for every $B\in \mathcal{B}\left( X\right) $ and $\mu \in \mathcal{M}_{1}.$
Hence $P^{\ast }$ transforms tight measures into tight measures. Then,
taking into account \cite[Corolary 2.4.1]{Sz3} and Remark \ref{remark}, the
proof given in \cite{Sz4} applies verbatim here to the general case and
shows that $\Phi $ has an invariant probablity measure when $X$ is a
complete metric space.
\end{example}

\section{Asymptotic stability}

In this section we prove the following criterion of stability which extends
\cite[Corollary 5.4]{SzWo} and, in the discrete case, \cite[Theorem 3.3]{Sz4}%
. The proof is an application of Proposition \ref{inv} and adapts the
arguments from \cite[Theorem 4.1]{LaYo} (see also \cite[Theorem 2]{SSU}).

\begin{theorem}
\label{Th1}Let $(P_{t})_{t\geq 0}$ be a Feller semigroup such that for any $%
\varphi \in C_{b}(X)$ the family $(P_{t}\varphi )_{t\geq 0}$ is
equicontinuous at every point $x$ of $X$. Then $(P_{t})_{t\geq 0}$ is
asymptotically stable if and only if there exists $z\in X$ such that for
every $\varepsilon >0$
\begin{equation}
\inf_{x\in X}\liminf_{t\rightarrow \infty }P_{t}^{\ast }\delta
_{x}(B(z,\varepsilon ))>0.  \label{sep}
\end{equation}
\end{theorem}

\begin{proof}
If $(P_{t})_{t\geq 0}$ is asymptotically stable with an invariant
probability measure $\mu _{\ast }$, then for all $z\in
%TCIMACRO{\TeXButton{supp}{\supp}}%
%BeginExpansion
\supp%
%EndExpansion
\mu _{\ast }$ and $x\in X,$%
\begin{equation*}
\liminf_{t\rightarrow \infty }P_{t}^{\ast }\delta _{x}(B(z,\varepsilon
))\geq \mu _{\ast }(B(z,\varepsilon ))>0.
\end{equation*}%
To prove the reverse implication, fix $\varepsilon >0,$ $\mu _{1},\mu
_{2}\in \mathcal{M}_{1}$ and $\varphi \in C_{b}(X).$ By assumption, there
exists $\delta >0$ such that%
\begin{equation*}
\left\vert P_{t}\varphi (x)-P_{t}\varphi (y)\right\vert <\varepsilon
\end{equation*}%
for $t\in \lbrack 0,\infty )$ and $x,y\in B(z,\delta ).$ By (\ref{sep}),
there exists $\alpha >0$ such that%
\begin{equation*}
\liminf_{t\rightarrow \infty }P_{t}^{\ast }\delta _{x}(B(z,\delta ))>\alpha
\end{equation*}%
for every $x\in X.$ Then, by Fatou's lemma,
\begin{equation}
\liminf_{t\rightarrow \infty }P_{t}^{\ast }\nu (B(z,\delta ))\geq
\int_{X}\liminf_{t\rightarrow \infty }P_{t}^{\ast }\delta _{x}(B(z,\delta
))\nu (dx)>\alpha  \label{sep1}
\end{equation}%
for every $\nu \in \mathcal{M}_{1}.$ We shall define a sequence $%
(t_{k})\subset T$ and four sequences of probability measures $(\mu
_{1}^{k}),(\mu _{2}^{k}),(\nu _{1}^{k}),(\nu _{2}^{k})$ by induction. Let $%
t_{0}=0,\mu _{1}^{0}=\nu _{1}^{0}=\mu _{1},\mu _{2}^{0}=\nu _{2}^{0}=\mu
_{2}.$ Suppose that we have chosen $t_{k-1},\mu _{1}^{k-1},\mu
_{2}^{k-1},\nu _{1}^{k-1},\nu _{2}^{k-1}$ $(k\geq 1).$ According to (\ref%
{sep1}), take $t_{k}$ such that%
\begin{equation}
P_{t_{k}}^{\ast }\mu _{i}^{k-1}(B(z,\delta ))>\alpha ,\ i=1,2,  \label{sep2}
\end{equation}%
and define%
\begin{align*}
\nu _{i}^{k}(A)& =\frac{P_{t_{k}}^{\ast }\mu _{i}^{k-1}(A\cap B(z,\delta ))}{%
P_{t_{k}}^{\ast }\mu _{i}^{k-1}(B(z,\delta ))}, \\
\mu _{i}^{k}(A)& =\frac{1}{1-\alpha }(P_{t_{k}}^{\ast }\mu
_{i}^{k-1}(A)-\alpha \nu _{i}^{k}(A))
\end{align*}%
for $A\in \mathcal{B}\left( X\right) ,$ $i=1,2.$ It is not difficult to see,
using (\ref{sep2}), that $\nu _{i}^{k},\mu _{i}^{k}$ are probability
measures. Now notice that
\begin{equation*}
P_{t_{k}}^{\ast }\mu _{i}^{k-1}=(1-\alpha )\mu _{i}^{k}+\alpha \nu _{i}^{k}
\end{equation*}%
and it is easy to verify by induction that%
\begin{equation*}
P_{t_{1}+...+t_{k}}^{\ast }\mu _{i}=\alpha P_{t_{2}+...+t_{k}}^{\ast }\nu
_{i}^{1}+\alpha (1-\alpha )P_{t_{3}+...+t_{k}}^{\ast }\nu
_{i}^{2}+...+\alpha (1-\alpha )^{k-1}\nu _{i}^{k}+(1-\alpha )^{k}\mu _{i}^{k}
\end{equation*}%
for $k\geq 1$ and $i=1,2.$ Hence%
\begin{align*}
& \left\vert \left\langle \varphi ,P_{t}^{\ast }\mu _{1}\right\rangle
-\left\langle \varphi ,P_{t}^{\ast }\mu _{2}\right\rangle \right\vert \\
=\,& \left\vert \left\langle P_{t-(t_{1}+...+t_{k})}\varphi
,P_{t_{1}+...+t_{k}}^{\ast }\mu _{1}\right\rangle -\left\langle
P_{t-(t_{1}+...+t_{k})}\varphi ,P_{t_{1}+...+t_{k}}^{\ast }\mu
_{2}\right\rangle \right\vert \\
\leq \,& \alpha \left\vert \left\langle P_{t-t_{1}}\varphi ,\nu _{1}^{1}-\nu
_{2}^{1}\right\rangle \right\vert +\alpha (1-\alpha )\left\vert \left\langle
P_{t-(t_{1}+t_{2})}\varphi ,\nu _{1}^{2}-\nu _{2}^{2}\right\rangle
\right\vert \\
& +...+\alpha (1-\alpha )^{k-1}\left\vert \left\langle
P_{t-(t_{1}+...+t_{k})}\varphi ,\nu _{1}^{k}-\nu _{2}^{k}\right\rangle
\right\vert +2(1-\alpha )^{k} \\
\leq \,& (\alpha +\alpha (1-\alpha )+...+\alpha (1-\alpha
)^{k-1})\sup_{x,y\in B(z,\delta ),t\geq 0}\left\vert P_{t}\varphi
(x)-P_{t}\varphi (y)\right\vert \\
& +2(1-\alpha )^{k}\leq \varepsilon +2(1-\alpha )^{k}
\end{align*}%
for $t\geq t_{1}+...+t_{k},$ $\varphi \in C_{b}(X)$ and $\mu _{1},\mu
_{2}\in \mathcal{M}_{1}.$ In particular,%
\begin{equation*}
\left\vert \left\langle \varphi ,P_{t}^{\ast }\mu _{1}\right\rangle
-\left\langle \varphi ,\mu _{\ast }\right\rangle \right\vert \leq
\varepsilon +2(1-\alpha )^{k},
\end{equation*}%
where $\mu _{\ast }$ is an invariant measure from Proposition \ref{inv}.
Letting $\varepsilon \rightarrow 0,k\rightarrow \infty ,$ yields%
\begin{equation*}
\lim_{t\rightarrow \infty }\left\vert \left\langle \varphi ,P_{t}^{\ast }\mu
_{1}\right\rangle -\left\langle \varphi ,\mu _{\ast }\right\rangle
\right\vert =0
\end{equation*}%
which shows that $(P_{t})_{t\geq 0}$ is asymptotically stable.
\end{proof}

As it was noted in Remark \ref{remark}, we can weaken the assumption to the
plain e-property in the above theorem when $X$ is complete and each $%
Q_{n}\delta _{z}$ is tight. A sufficient condition for Polish spaces was
given in Theorem 1 of \cite{BKS}. We show that it is also a necessary
condition, which leads to the following theorem.

\begin{theorem}
Let $(P_{t})_{t\geq 0}$ be a Feller semigroup on a Polish space $X$ with the
e-property. The following conditions are equivalent:

\begin{enumerate}
\item[(i)] There exists $z\in X$ such that for every $\varepsilon >0$
\begin{equation*}
\inf_{x\in X}\liminf_{t\rightarrow \infty }P_{t}^{\ast }\delta
_{x}(B(z,\varepsilon ))>0,
\end{equation*}

\item[(ii)] For any $\varepsilon >0$ and bounded set $A\subset X$ there is a
bounded Borel set $B\subset X$ such that%
\begin{equation*}
\sup_{\mu \in \mathcal{M}_{1}^{A}}\limsup_{t\rightarrow \infty }\frac{1}{t}%
\int_{0}^{t}P_{s}^{\ast }\mu (B)ds>1-\epsilon
\end{equation*}%
and, there is $z\in X$ such that for any $\varepsilon >0$ and bounded set $%
A\subset X$
\begin{equation*}
\inf_{\mu _{1},\mu _{2}\in \mathcal{M}_{1}^{A}}\sup_{t>0}\min \{P_{t}^{\ast
}\mu _{1}(B(z,\varepsilon )),P_{t}^{\ast }\mu _{2}(B(z,\varepsilon ))\}>0,
\end{equation*}

\item[(iii)] $(P_{t})_{t\geq 0}$ is asymptotically stable.
\end{enumerate}
\end{theorem}

\begin{proof}
By Theorem \ref{Th1} and Remark \ref{remark}, (i) and (iii) are equivalent.
By \cite[Theorem 1]{BKS} (ii) implies (iii), so it suffices to show that
(iii)$\Rightarrow $(ii). Suppose that $(P_{t})_{t\geq 0}$ is asymptotically
stable with an invariant probability measure $\mu _{\ast }.$ Then, for any $%
\varepsilon >0$, there exists a compact set $K$ such that
\begin{equation*}
\limsup_{t\rightarrow \infty }\frac{1}{t}\int_{0}^{t}P_{s}^{\ast }\mu
(K)ds\geq \liminf_{t\rightarrow \infty }P_{t}^{\ast }\mu (K)\geq \mu _{\ast
}(K)>1-\varepsilon
\end{equation*}%
and the first condition holds with $B=K$ for any $\mu \in \mathcal{M}_{1}.$
Likewise,%
\begin{equation*}
\liminf_{t\rightarrow \infty }P_{t}^{\ast }\mu (B(z,\varepsilon ))\geq \mu
_{\ast }(B(z,\varepsilon ))>0
\end{equation*}%
for any $z\in
%TCIMACRO{\TeXButton{supp}{\supp}}%
%BeginExpansion
\supp%
%EndExpansion
\mu _{\ast }$ and $\mu \in \mathcal{M}_{1}$, which is stronger than we need.
\end{proof}

Since our approach also works in the discrete case, we obtain the following
theorem which extends \cite[Theorem 3.3]{Sz4}.

\begin{theorem}
\label{stab_dis}Let $X$ be a metric space and $\pi :X\times \mathcal{B}%
\left( X\right) \rightarrow \lbrack 0,1]$ a transition function for a
discrete-time Markov chain $\Phi $. Assume that there exists $z\in X$ such
that for every $\varepsilon >0$ there exists $\alpha >0$ such that%
\begin{equation*}
\liminf_{n\rightarrow \infty }\pi ^{n}(x,B(z,\varepsilon ))\geq \alpha
\end{equation*}%
for every $x\in X.$ If the family $\left\{ \int_{X}\varphi (y)\pi
^{n}(z,dy):n\in \mathbb{N}\right\} $ is equicontinuous for every $\varphi
\in C_{b}(X)$, then $\Phi $ has a unique invariant probability measure $\mu
_{\ast }$ and $\int_{X}\pi ^{n}(x,\cdot )\mu (dx)$ converges weakly to $\mu
_{\ast }$ as $n\rightarrow \infty $ for every $\mu \in \mathcal{M}_{1}.$
\end{theorem}

Finally, let us point out that if $(P_{t})_{t\geq 0}$ is an asymptotically
stable Feller semigroup with the e-property for bounded continuous
functions, then it converges weakly to $\mu _{\ast },$ uniformly on compact
sets.

\begin{theorem}
\label{uniform}Let $K$ be a compact subset of a metric space $X$ and $%
(P_{t})_{t\geq 0}$ an asymptotically stable Feller semigroup such that for
any $\varphi \in C_{b}(X)$ the family $(P_{t}\varphi )_{t\geq 0}$ is
equicontinuous at every point $x$ of $X$. Then there exists $\mu _{\ast }\in
\mathcal{M}_{1}$ such that
\begin{equation*}
\lim_{t\rightarrow \infty }\sup_{x\in K}\left\vert P_{t}\varphi
(x)-\left\langle \varphi ,\mu _{\ast }\right\rangle \right\vert =0
\end{equation*}%
for every $\varphi \in C_{b}(X).$ If a metric space $X$ is complete and $%
P_{t}^{\ast }\delta _{x}$ is tight for every $t\geq 0$ and $x\in K$, then
the e-property for bounded continuous functions may be relaxed to the
(plain) e-property.
\end{theorem}

\begin{proof}
Fix $\varphi \in C_{b}(X).$ By assumption, there exists $\mu _{\ast }\in
\mathcal{M}_{1}$ such that
\begin{equation*}
\lim_{t\rightarrow \infty }\left\langle \varphi ,P_{t}^{\ast }\delta
_{x}\right\rangle =\lim_{t\rightarrow \infty }P_{t}\varphi (x)=\left\langle
\varphi ,\mu _{\ast }\right\rangle
\end{equation*}%
for every $x\in X.$ Fix $\varepsilon >0.$ Since $(P_{t}\varphi )_{t\geq 0}$
is equicontinuous at a point $x\in X$, there exists $\delta _{x}\in
(0,\varepsilon )$ such that%
\begin{equation*}
\left\vert P_{t}\varphi (x)-P_{t}\varphi (y)\right\vert <\varepsilon
\end{equation*}%
for $t\in \lbrack 0,\infty )$ and $y\in B(x,\delta _{x}).$ Since $K$ is
compact, there exist $x_{1},...,x_{n}\in K$ such that for every $y\in K$
there exists $1\leq i\leq n$ such that $y\in B(x_{i},\delta _{x_{i}})$ and
hence $\left\vert P_{t}\varphi (x_{i})-P_{t}\varphi (y)\right\vert
<\varepsilon .$ Thus
\begin{equation*}
\limsup_{t\rightarrow \infty }\sup_{y\in K}\left\vert P_{t}\varphi
(y)-\left\langle \varphi ,\mu _{\ast }\right\rangle \right\vert \leq
\varepsilon +\lim_{t\rightarrow \infty }\sup_{1\leq i\leq n}\left\vert
P_{t}\varphi (x_{i})-\left\langle \varphi ,\mu _{\ast }\right\rangle
\right\vert \leq \varepsilon
\end{equation*}%
and letting $\varepsilon \rightarrow 0$ proves the first half of the theorem.

Now suppose that $X$ is complete and each $P_{t}^{\ast }\delta _{x}$ is
tight. We show that the family $\{P_{t}^{\ast }\delta _{x}:t\geq 0,x\in K\}$
is uniformly tight. To this end, choose a sequence $\{P_{t_{n}}^{\ast
}\delta _{x_{n}}\}_{n\in \mathbb{N}},$ $x_{n}\in K$ and let a subsequence $%
(x_{n_{k}})$ tends to an $x_{0}\in K$ Since $(P_{t})_{t\geq 0}$ is an
asymptotically stable semigroup with the e-property,
\begin{equation*}
\lim_{k\rightarrow \infty }\left\langle \psi ,P_{t_{n_{k}}}^{\ast }\delta
_{x_{n_{k}}}\right\rangle =\lim_{k\rightarrow \infty }P_{t_{n_{k}}}\psi
(x_{n_{k}})=\lim_{k\rightarrow \infty }P_{t_{n_{k}}}\psi
(x_{0})=\left\langle \psi ,\mu _{\ast }\right\rangle
\end{equation*}%
for every $\psi \in \mathrm{Lip}_{b}(X).$ Since $\mathrm{Lip}_{b}(X)$ is
convergence determining (see, e.g., \cite[Theorem 3.3.1]{EtKu}), $%
P_{t_{n_{k}}}^{\ast }\delta _{x_{n_{k}}}$ is weakly convergent to $\mu
_{\ast }.$ Hence $\{P_{t}^{\ast }\delta _{x}:t\geq 0,x\in K\}$ is uniformly
tight (see, e.g., \cite[Theorem 8.6.2]{Bo}). Thus, for every $\varepsilon >0$
there exists a compact set $K_{\varepsilon }$ such that $P_{t}^{\ast }\delta
_{x}(X\setminus K_{\varepsilon })<\varepsilon $ for every $t\geq 0$ and $%
x\in K.$ Fix $\varepsilon >0$ and $\varphi \in C_{b}(X).$ Since $\mathrm{Lip}%
(K_{\varepsilon })$ is dense in $C(K)$ for $\left\Vert \cdot \right\Vert
_{\infty },$ there exists $\psi \in \mathrm{Lip}(K_{\varepsilon })$ such
that $\left\vert \varphi (x)-\psi (x)\right\vert <\varepsilon $ for $x\in
K_{\varepsilon }.$ By the Kirszbraun-McShane-Whitney theorem (see, e.g.,
\cite[Prop. 11.2.3]{Du}), we can extend it to $\bar{\psi}\in \mathrm{Lip}%
_{b}(X).$ Thus%
\begin{align*}
& \left\vert P_{t}\varphi (x)-P_{t}\bar{\psi}(x)\right\vert =\left\vert
\left\langle \varphi ,P_{t}^{\ast }\delta _{x}\right\rangle -\left\langle
\bar{\psi},P_{t}^{\ast }\delta _{x}\right\rangle \right\vert \\
\leq & \,\int_{X\setminus K_{\varepsilon }}\left\vert \varphi (y)-\bar{\psi}%
(y)\right\vert P_{t}^{\ast }\delta _{x}(dy)+\int_{K_{\varepsilon
}}\left\vert \varphi (y)-\bar{\psi}(y)\right\vert P_{t}^{\ast }\delta
_{x}(dy) \\
\leq & \ \varepsilon (\left\Vert \varphi -\bar{\psi}\right\Vert +1)
\end{align*}%
for every $x\in X$ and $t\geq 0.$ From the e-property, $(P_{t}\bar{\psi}%
)_{t\geq 0}$ is equicontiuous at every point $x$ of $X$ and by the above
argument, $(P_{t}\varphi )_{t\geq 0}$ is equicontiuous at every point, too.
We can now repeat the arguments from the first part of the proof.
\end{proof}

Since the above result is also valid in the discrete case, we obtain the
following generalization of \cite[Theorem 3.3]{CzHo}.

\begin{theorem}
\label{undiscrete}Let $K$ be a compact subset of a metric space. Under the
assumptions of Theorem \ref{stab_dis},%
\begin{equation*}
\lim_{n\rightarrow \infty }\sup_{%
%TCIMACRO{\TeXButton{supp}{\supp}}%
%BeginExpansion
\supp%
%EndExpansion
\mu \subset K}\left\vert \left\langle \varphi ,\int_{X}\pi ^{n}(x,\cdot )\mu
(dx)\right\rangle -\left\langle \varphi ,\mu _{\ast }\right\rangle
\right\vert =0.
\end{equation*}%
In particular, $\pi ^{n}(x,\cdot )$ converges weakly to $\mu _{\ast },$
uniformly on compact subsets of $X.$
\end{theorem}

Notice that if $X$ is complete and each $\pi ^{n}(x,\cdot )$ is tight, it
follows from Remark \ref{remark} and the second part of Theorem \ref{uniform}
that it is sufficient to assume that the family $\left\{ \int_{X}\varphi
(y)\pi ^{n}(z,dy):n\in \mathbb{N}\right\} $ is equicontinuous for every $%
\varphi \in \mathrm{Lip}_{b}(X)$ rather than for every bounded continuous $%
\varphi .$ It enables us to extend Theorem 6.1 in \cite{CzHo}. Let $X$ be a
complete metric space and consider on $X$ an iterated function system $%
(w,p)_{N}=(w_{1},...,w_{N},p_{1},...,p_{N})$ as in Example \ref{1}. Let $%
\mathcal{R}$ denote the collection of all couples $(r,\omega )$ of
nondecreasing and continuous at $0$ functions from $[0,\infty )$ into $%
[0,\infty )$ with $r(0)=\omega (0)=0$ such that the series $%
\sum\nolimits_{n=1}^{\infty }\omega (r^{n}(t))$ is convergent and $r(t)<t$
for some $T>0$ and every $0<t<T.$

\begin{theorem}
Let the iterated function system $(w,p)_{N}$ on a complete metric space $X$
satisfy%
\begin{align*}
\sum_{i=1}^{N}p_{i}(x)\rho (w_{i}(x),w_{i}(y))& \leq r(\rho (x,y)), \\
\sum_{i=1}^{N}\left\vert p_{i}(x)-p_{i}(y)\right\vert & \leq \omega (\rho
(x,y))
\end{align*}%
for all $x,y\in X$ and a couple of concave functions $(r,\omega )\in
\mathcal{R}$. Furthermore, suppose that there exist $\hat{x}\in X$ with $%
\min_{1\leq i\leq N}p_{i}(\hat{x})>0$ and $k\in \{1,...,N\}$ such that $%
w_{k} $ is a contraction and $\inf_{x\in X}p_{k}(x)>0.$ Then the transition
operator%
\begin{equation*}
Pf(x)=\sum_{i=1}^{N}f(w_{i}(x))p_{i}(x)
\end{equation*}%
corresponding to $(w,p)_{N}$ is asymptotically stable and $((P^{\ast
})^{n}\delta _{x}(\cdot )):=(P^{n}\mathbf{1}_{(\cdot )}(x))$ converges
weakly to a unique $\mu _{\ast }\in \mathcal{M}_{1}$ uniformly with respect
to $x$ on compact subsets of $X.$
\end{theorem}

\begin{proof}
We can prove analogously to Theorem 5.1 in \cite{CzHo} that the family $%
\left\{ P^{n}\varphi :n\in \mathbb{N}\right\} $ is equicontinuous at every
point $x\in X$ for any $\varphi \in \mathrm{Lip}_{b}(X).$ Now, the arguments
of \cite[Theorem 6.1]{CzHo}, Theorem \ref{stab_dis} and Remark \ref{remark}
show that $P$ is asymptotically stable. Since $X$ is complete and $P^{\ast }$
transforms tight measures into tight measures, analysis similar to that in
the proof of Theorem \ref{uniform} shows that $\left\{ P^{n}\varphi :n\in
\mathbb{N}\right\} $ is also equicontinuous for any $\varphi \in C_{b}(X).$
We can now apply Theorem \ref{undiscrete}.
\end{proof}

\begin{remark}
Recall \cite{Sz3} that a Markov operator $P^{\ast }:\mathcal{M}%
_{1}\rightarrow \mathcal{M}_{1}$ is concentrating if for every $\varepsilon
>0$ there exists a Borel set $C\subset X$ with $\diam C<\varepsilon $ such
that
\begin{equation*}
\inf_{\mu \in \mathcal{M}_{1}}\liminf_{n\rightarrow \infty }(P^{\ast
})^{n}(C)>0.
\end{equation*}%
It was proved in \cite[Theorem 3.3]{Sz1} that a Markov operator defined on a
Polish space, which is concentrating and nonexpansive in the Fortet-Mourier
norm, is asymptotically stable. It is not clear whether a similar result is
also valid in any metric space.
\end{remark}

\begin{ack}
The authors are grateful to the referee for several helpful improvements.
\end{ack}

\end{document}